\documentclass{amsart}
\usepackage{amsfonts}
\usepackage{amssymb}
\usepackage{amscd}
\usepackage{tikz-cd}
\usepackage{mathrsfs}
\usepackage{color}
\newtheorem{theorem}{Theorem}[section]
\newtheorem{lemma}[theorem]{Lemma}
\newtheorem{proposition}[theorem]{Proposition}

\theoremstyle{definition}
\newtheorem{definition}[theorem]{Definition}

\theoremstyle{remark}

\numberwithin{equation}{section}

\begin{document}

\title{The Lichtenbaum conjecture for abelian extensions of  imaginary quadratic fields }

\author{Chaochao Sun}
\address{Chaochao Sun, School of  Mathematics and Statistics, Linyi University, Linyi,
China 276005}

\email{sunuso@163.com}

%\author{Xin Zhang}
%\address{School of Mathematic \&Physics, Qingdao University of Science \& Technology, Qingdao 266071, China}
%\email{}
%\author{Long Zhang}
%\address{Long Zhang, School of Mathematics and Statistics, Qingdao University, Qingdao, China, 266071}
%\email{zhanglong\_note@hotmail.com}

%\thanks{The research was supported  by NSFC (No.11601211).}



%    General info
\subjclass[2000]{Primary  ; Secondary }

%\date{January 1, 2001 and, in revised form, June 22, 2001.}



\begin{abstract}
In this paper, we prove the cohomological Lichtenbaum conjecture of abelian extensions of  imaginary quadratic fields up to a finite set of bad primes.
\end{abstract}

\maketitle

%% The correct journal style for \specialsection is all uppercase; a known bug
%% in amsart.cls prevents this, so input must be uppercase until it is fixed.
%\specialsection*{This is a Special Section Head}
%%%%%%%%%%%%%%%%%%%%%%%%%%%%%%%%%%%%%%%%%%%%%%%%%%%%%%%%%%%%%%%%%%%%%%%%

%%%%%%%%%%%%%%%%%%%%%%%%%%%%%%%%%%%%%%%%%%%%%%%%%%%%%%%%%%%%%%%%%%%%%%%%

\section{Introduction}

In this paper, we prove the Lichtenbaum conjecture of abelian extensions of  imaginary quadratic fields.  For this we compare the elliptic unites  with the etale cohomology by a similiar ``Soul\'{e} map''. This method has been used by Kings (see\cite{Kin01}) to prove the Tamagawa conjecture for CM elliptic curves.

The Lichtenbaum conjecture on special values of zeta functions of number fields  at negative integers (see\cite{Li}) stems from generalizing the analytic class number formula. Let $F$ be a number field, $m\geq 2$. Let $R_m(F)$ be the Beilinson regulator of $F$. The cohomological Lichtenbaum conjecture is, up to sign and up to powers of 2,
\[ \zeta_F(1-m)^*=R_m(F)\prod_p \frac{\# H^2(\mathcal{O}_F[1/p], \mathbb{Z}_p(m))}{\# H^1(\mathcal{O}_F[1/p], \mathbb{Z}_p(m))_{\mbox{tors}}}. \]
In particular, the cohomological Lichtenbaum conjecture holds for abelian number fields $F$.

Borel in \cite{Bo2} made a breakthrough on it, proving that the first non-vanishing coefficient of Talyor series of zeta function at negative integers is equal to Borel regulator up to a rational factor. Beilinson in \cite{Be} formulated far reaching conjecture describing special values of $L$ function of motives in terms of Beilinson regulator. In number field case, Borel regulator map is twice of Beilinson regulator map(see \cite{G}Thm.10.9).

In order to understand the rational factor of special value of $L$ functions of motives, Bloch and Kato in \cite{BK} formulated a conjecture, that is the Bloch-Kato conjecture of special value of $L$ function, which is expressed in terms of Haar measures and Tamagawa numbers. The Bloch-Kato conjecture for Artin motive implies the cohomological Lichtenbaum conjecture(see \cite[Theorem1.4.1]{HK03} ).
 Fontaine and Perrin-Riou \cite{FR} proposed another formulation in terms of determinants of perfect complexes and Galois cohomology. The two formulations,  are equivalent at least when the field of coefficients is $\mathbb{Q}$, which is shown by T. Nguyen Quang Do in \cite{N}. The abelian cohomological Lichtenbaum conjecture was first proved in \cite{KNF}, but with some errors which were subsequently corrected in the appendix of \cite{BN}. J.Coates \cite{C} reproved the Bloch-Kato conjecture of $\mathbb{Q}$ in a view of Iwasawa theory. Since rational field $\mathbb{Q}$ and imaginary quadratic field behave similar in some ways, such as cyclotomic units and elliptic units, this lead us to consider the Lichtenbaum conjecture of abelian extensions of imaginary quadratic fields in a similar way.

In the  sections 2-4, we recall the known results that we need. Section 5 is main work of us to calculate an index by Iwasawa theory of imaginary quadratic fields. Section 6  collect all the foregoing results to obtain the Lichtenbaum conjecture of abelian extension of imaginary quadratic fields.

\textbf{Acknowledgements} I thank ?? for reading the manuscript and giving many valuable suggestions.

\section{Special value of  $L$-funcion of Imaginary quadratic field}

Let $K$ be an imaginary quadratic field. Take an abelian extension $F$ of $K$ with conductor $\mathfrak{f}$. let $E$ be an elliptic curve over $F$, with complex multiplication by $\mathcal{O}_K$, satisfying $F(E_{\text{tor}})$ being an abelian extension of $K$. This implies $F\supset K(1)$. The condition
$F(E_{\text{tor}})$ being an abelian extension of $K$ is equivalent to that there exists a Gr\"{o}ssencharacter $\varphi$ of type $(1,0)$ of $K$ such that the Gr\"{o}ssencharacter $\psi$ of $E$ over $F$ can be written as
$$\psi=\varphi\circ N_{F/K},$$
where $N_{F/K}$ denotes the norm map from $F$ to $K$.

Let $\chi\neq 1$ be an irreducible character of
$\text{Gal}(F/K)$, and $\widetilde{\chi}$ the associated Hecke character mod $\mathfrak{f}$. Then the Artin $L$-function of character $\chi$ and the Hecke $L$-function of Hecke character $\widetilde{\chi}$ satisfy the identity
$$L(F/K,\chi,s)=\prod_{\mathfrak{p}\in S}\frac{1}{1-\chi(\varphi_{\mathfrak{P}})\mathfrak{N}(\mathfrak{p})^{-s}}L(\widetilde{\chi},s),$$
where $S=\{\mathfrak{p}|\mathfrak{f}\ |\ \chi(I_{\mathfrak{P}})=1\}$.

For simplicity, we take $F=K(\mathfrak{f})$, the ray class field of $K$ modulo $\mathfrak{f}$. Fix an embedding of $K(\mathfrak{f})$ into $\mathbb{C}$ such that $j(E)=j(\mathcal{O}_K)$. There is a unique isomorphism
$$\theta_E : \mathcal{O}_K \cong \text{End}_{K(\mathfrak{f})}(E)$$
such that $\theta_E^*(\alpha)=\alpha\omega$ for any invariant differential $\omega\in \Omega_{E/K(\mathfrak{f})}$ and for all $\alpha\in \mathcal{O}_K$.
Viewing $E$ as an elliptic curve over $\mathbb{C}$, we have $E\cong \mathbb{C}/\Gamma$ where the period lettice $\Gamma=\Omega \mathcal{O}_K$ for some $\Omega \in \mathbb{C}$.

Let $1\neq \mathfrak{m}|\mathfrak{f}$, and $\rho_{\mathfrak{m}}\in I_K$ be an id\`{e}le with ideal $\mathfrak{m}$. We can choose a $f_{\mathfrak{m}}\in K^*$ with
$$v_{\mathfrak{p}}(f_{\mathfrak{m}})\leq 0\ \text{if}\ \mathfrak{p}\nmid \mathfrak{m}\ \text{and} \ v_{\mathfrak{p}}(f_{\mathfrak{m}}^{-1}-(\rho_{\mathfrak{f}})_{\mathfrak{p}}^{-1})\geq0\ \text{if}\ \mathfrak{p}|\mathfrak{m}.$$
Let $G_{\mathfrak{f}}=\text{Gal}(F/K)$.  For $g\in G_{\mathfrak{f}}$, let ${}^{g}E$ be the curve obtained by base change
\[
\begin{CD}
 {}^{g}E @>>>E\\
@VVV @VVV  \\
F @>g>> F    \\
\end{CD}\quad .
\]

Let $\mathcal{A}=R_{K(\mathfrak{f})/K}E$ be the Weil restriction of the elliptic curve. $\mathcal{A}$ is an abelian variety over $K$ with CM by a semisimple $K$-algebra $T$ and Serre-Tate character $\varphi_{\mathcal{A}}$. Deninger proves that any Hecke character, $\varphi$ of weight $w>0$ is of
the form $\prod_{i=1}^w\varphi_{\lambda_i}$ where $\lambda_i\in \text{Hom}(T,\mathbb{C})$ and $\varphi_{\lambda_i}=\lambda_i\circ \varphi_{\mathcal{A}}$(\cite{D}Prop.1.3.1).

Let $[\ ]: \mathbb{C}\to \mathbb{C}/\Gamma_g$ be the projection and observe that ${}^{g}E(\mathbb{C})_{\mathfrak{f}}=\mathfrak{f}^{-1}\Gamma_g/\Gamma_g$.
Let $\beta=([\Omega f_{\mathfrak{f}}^{-1}])$ be the divisor associated with the $\mathfrak{f}$-torsion point $[\Omega f_{\mathfrak{f}}^{-1}]$ on $E$ which is rational over $K(\mathfrak{f})$.

Twisting a complex character $\eta$ by the norm character gives a Hecke character of $K$ of weight $2$
\[ \eta N_{K/\mathbb{Q}}= \varphi_{\lambda_1}\varphi_{\lambda_2}.\]
By functoriality, we can increase $\mathfrak{f}$ so that it is a multiple of the conductors of $\varphi_{\lambda_1}$ and $\varphi_{\lambda_2}$.
Fix a set of ideals $\{ \mathfrak{b}_g\subseteq \mathcal{O}_K\}_{g\in G_\mathfrak{f}}$ with Artin symbol $(\mathfrak{b}_g, K(\mathfrak{f})/K)=g\in G_\mathfrak{f}$. For integral ideals $\mathfrak{a}$ of $K$ prime to the conductors of $\varphi_{\lambda_1}$ and $\varphi_{\lambda_2}$, define
$\Lambda(\mathfrak{a})\in K(\mathfrak{f})^\times$ by
\[ \varphi_\mathcal{A}(\mathfrak{a})^*\omega^{\sigma_\mathfrak{a}}=\Lambda(\mathfrak{a})\omega \]
where $\omega\in H^0(E,\Omega^1)$ has period lattice $\Gamma$, $\varphi_\mathcal{A}(\mathfrak{a})\in T^\times$ is viewed as an isogeny
$E\to {}^{\sigma_\mathfrak{a}}E$ and $\sigma_\mathfrak{a}$ is the Artin automorphism of $\mathfrak{a}$. For all $g\in G_\mathfrak{f}$, $\omega^g$ has period lattice $\Gamma_g$ and we can identify ${}^gE(\mathbb{C})$ with $\mathbb{C}/\Gamma_g$ via Abel-Jacobi map to obtain a divisor
\[ \beta_g=(\Lambda(\mathfrak{b}_g))\Omega f_\mathfrak{m}^{-1} \]
on ${}^g E(K(\mathfrak{f}))$ with $G_\mathfrak{f}$ action given by ${}^h\beta_g=\beta_{hg}$.

For the special values of these Artin $L$-functions at negative integers, there is a proposition (\cite{D}(3.4),\cite{J})
\begin{proposition}
The $L$-series $L(\chi,s)$ has a first order zero for every $s=-j<0$, and the special value is given by the formula
$$L'(\chi,-j)=(-1)^{j}\frac{\Phi(\mathfrak{f}_\chi)(j!)^2}{\Phi(\mathfrak{f})}\bigg(\frac{\sqrt{d_K}\mathcal{N}\mathfrak{f}_\chi}{2\pi i}\bigg)^j\chi(\rho_{\mathfrak{f}_\chi})\sum_{g\in G_{\mathfrak{f}}}\chi(g)A(\Gamma_g)^{1+j}\mathcal{M}_j(\beta_g),$$
where $\Phi$ is the totient function and $d_K$ is the discriminant of $K$. For any $\mathbb{Z}$-basis of $\Gamma_g$ with $Im(v/u)>0, A(\Gamma_g)=(\bar{u}v-\bar{v}u)/2\pi i$, and for divisors on the $\mathfrak{f}_\chi$-torsion of $E$, $\mathcal{M}_j$ is defined by linearity from
$$\mathcal{M}_j(x)=\sum_{0\neq\gamma\in \Gamma_g}\frac{(x,\gamma)_g}{|\gamma|^{2(1+j)}}, \ x\in E_{\mathfrak{f}_\chi}$$
with the Pontryagin $(,)_g: \mathbb{C}/\Gamma_g \times \Gamma_g\to U(1)$ given by $(z,\gamma)_g=\exp(A(\Gamma_g)^{-1}(z\bar{\gamma}-\bar{z}\gamma))$.
\end{proposition}

\section{Motivic elements}

Let $X$ be a smooth variety over $\mathbb{Q}$. The motivic cohomology of $X$ is the bigraded family of rational vector spaces
$$H_\mathcal{M}^j(X,j)=(K_{2j-i}(X)\otimes\mathbb{Q})^{(j)},$$
where the subscript $(j)$ denotes the simultaneous eigenspace of all Adams operators $\psi^k, k\geq1$ belonging to the eigenvalue $k^j$. The groups $K_*(X)$ are the usual Quillen $K$-groups.

    The original construction of Eisenstein classes in motivic cohomology was given by Beilinson in\cite{Be}. The classes obtained are called the Eisenstein symbols. Deninger \cite{D89}constructs motivic elements from the divisor $\beta$ using a variation of the Eisention symbol
$$\mathcal{E}_{\mathcal{M}}^k: \mathbb{Q}[E]^0 \to H_{\mathcal{M}}^{k+1}(E^k, k+1)$$
which is defined only for divisors of degree 0. The Eisenstein symbol used in the work of Huber, Kings and Scholl\cite{HK99,Kin01,Sch98} are related with Deninger's. Their relation has been demonstrated by Johnson in \cite{J}. Choose an integer $N\geq2$ and define a degree 0 divisor on ${}^gE(\mathbb{C})$
$$\alpha_g=N^2(0)-\sum_{p\in{}^gE(\mathbb{C})_N}(p).$$
Deninger shows that
\begin{proposition} The degree zero divisor
$$\beta'_g=\beta_g-(\deg \beta_g)(0)+\frac{\deg \beta_g}{N^2}\left(1-\frac{1}{N^{4+2j}}\right)^{-1}\alpha_g$$
is defined over ${}^gE(K(\mathfrak{f}))$ with ${}^h\beta'_g=\beta'_{hg}$ and
$$\mathcal{M}_j(\beta'_g)=\mathcal{M}_j(\beta'_g).$$
\end{proposition}

Johnson \cite{J} extended Deninger's Eisenstein symbol to any degree, which is necessary in order to have the desired shape in the $\ell$-adic computation.
\begin{proposition}
  For $k>0$, there is a variation of the Eisenstein symbol $\mathcal{E}is^k_\mathcal{M}: \mathbb{Q}[E[\mathfrak{f}]\setminus 0]\to H^{k+1}_\mathcal{M}(E^k,k+1)$ which is defined for divisors of any degree. Moreover,
  $$\mathcal{E}is(\beta')=\mathcal{E}is(\beta).$$
\end{proposition}

Deninger constructs the Kronecker map, which is a projector given by

\[ \begin{tikzcd}
		H_{\mathcal{M}}^{2j+1}(E^{2j}, 2j+1) \arrow[rr, "(\delta^j)^*"] \arrow[rrd, "\mathcal{K}_\mathcal{M}"'] & & H_{\mathcal{M}}^{2j+1}(E^j, 2j+1) \arrow[d, "\text{pr}_*"] \\
	&	 & H^1_{\mathcal{M}}(\text{Spec}(K(\mathfrak{f})),j+1),
	\end{tikzcd} \]
where $\delta=(\text{id}, \theta_E(\sqrt{d_K})): E\to E\times_{K(\mathfrak{f})} E$ and  this $j$-times gives $\delta^j: E^j\to E^{2j}$, pr is the projection $E^j \to \text{Spec}(K(\mathfrak{f}))$. Then the Kronecker map is
$$\mathcal{K}_\mathcal{M}=\text{pr}_*\circ (\delta^j)^*.$$

Recall that the regulator $\rho_\infty$ is an isomorphism
$$K_{2j+1}(\mathcal{O}_{K(\mathfrak{f})})\otimes_\mathbb{Z}\mathbb{R}\xrightarrow{\rho_\infty} \left(\bigoplus_{\sigma\in \mathcal{T}}\mathbb{C}/\mathbb{R}\cdot(2\pi i)^{j+1}\cdot\sigma \right)^+,$$
where $\mathcal{T}=\text{Hom}(K(\mathfrak{f}),\mathbb{C})$.

Thus Johnson defines
$$\xi_\mathfrak{m}(j):=\mathcal{K}_\mathcal{M}\mathcal{E}is^{2j}(\rho_\mathfrak{m}\cdot\beta).$$
Note that the definition does not construct motivic element for small $\mathfrak{m}$. The following lemma\cite{J}(Lem.3.1.8) atones for this omission.

\begin{lemma}
  \[
  w_\mathfrak{m}/w_{\mathfrak{pm}}\text{Tr}_{K(\mathfrak{pm})/K(\mathfrak{m})}\xi_\mathfrak{mp}(j)=
  \begin{cases}
    \xi_\mathfrak{m}(j), & \mbox{if }\ \mathfrak{p}\mid\mathfrak{m}\neq1 \\
    (1-\text{Fr}_\mathfrak{p}^{-1})\xi_\mathfrak{m}(j), & \mbox{if}\ \ \mathfrak{p}\nmid\mathfrak{m}\neq1.
  \end{cases}
  \]
\end{lemma}
With this lemma one can define motivic elements for small ideals via the Trace map. In particular, for any prime $\mathfrak{q}$ of $K$ with $w_\mathfrak{q}=1$, one define
$${}_\mathfrak{q}\xi_1(j):=(1-\text{Fr}_\mathfrak{q}^{-1})^{-1}w_K \text{Tr}_{K(\mathfrak{q})/K(1)}\xi_\mathfrak{q}(j)$$
for a family of motivic element at level 1.

Johnson proves the following theorem which is a modification of a result of Deninger\cite{D}.
\begin{theorem}\label{Thm3.4}
  For every ideal $1\neq\mathfrak{m}\mid\mathfrak{f}$, there are motivic elements
  $$\xi_\mathfrak{m}(j)\in H^1_\mathcal{M}(K(\mathfrak{f}),j+1)$$
  with the property that if $\chi$ is a rational character of $G_\mathfrak{f}$ of conductor $\mathfrak{m}$, then
  $$e_\chi(\rho_\infty(\xi_\mathfrak{m}(j)))=\frac{(2\mathcal{N}\mathfrak{m})^{j-1}\Phi(\mathfrak{f})}{(-1)^{1-j}(2j)!\Phi(\mathfrak{m})}
  L'(\chi,-j)\eta_\mathbb{Q},$$
  where $\eta_\mathbb{Q}$ is a basis of the $\chi$-component of $(H_B^0(\text{Spec}(F)(\mathbb{C}),\mathbb{Q}(j))^+)^*,$ and
  $$\Phi(\mathfrak{f})=|(\mathcal{O}_K/\mathfrak{f})^\times|=\mathcal{N}\mathfrak{f}\prod_{\mathfrak{p}\mid\mathfrak{f}}(1-\mathcal{N}\mathfrak{p}^{-1}).$$
  Moreover, these elements forms a norm compatible system, and for $\mathfrak{m}=1$ we have a family of elements, ${}_\mathfrak{q}\xi_1(j)$, indexed by the primes of $K$ which are defined via the norm map and satisfying the above formula for any choice of $\mathfrak{q}$ with $w_\mathfrak{q}=1.$
\end{theorem}

\section{$\ell$-adic regulator of motivic element}
This section need to introduce elliptic units, so let us review the construction of them. Let $L=\mathbb{Z}w_1+\mathbb{Z}w_2$ be a lattice in $\mathbb{C}$ such that $\tau:=w_1/w_2$ has positive imaginary part. The Dedekind eta-function is defined as
$$\eta(\tau)=e^{\frac{\pi i\tau}{12}}\prod_{n=1}^{\infty}(1-q_\tau^n); \quad q_\tau=e^{2\pi i \tau}.$$
Put
$$\eta(w_1,w_2)=w_2^{-1}2\pi\eta(w_1/w_2)^2.$$
This function depends on the choice of basis but
$$\Delta(L)=\Delta(\tau)=\eta(w_1,w_2)^{12}$$
does not. Define a theta-function
$$\theta(z,\tau)=ie^{\frac{\pi i z}{2}\frac{z-\bar{z}}{\tau-\bar{\tau}}}q_\tau^{1/12}q_z^{-1/2}(1-q_z)\prod_{n=1}^{\infty}(1-q_zq_\tau^n)(1-q_z^{-1}q_\tau^n),$$
where $q_z=e^{2\pi i z}$ and
$$\theta(z;w_1,w_2)=\theta(z/w_2,w_1/w_2).$$
For any pair lattices $L\subseteq L'$ of index prime to 6 with oriented bases $\omega=(w_1,w_2)$ and $\omega'=(w'_1,w'_2)$, Robert \cite{Rob92} shows that there exists a unique choice of 12-th root of unity $C(\omega,\omega')$ such that the functions
$$\delta(L,L'):=C(\omega,\omega')\eta(w_1,w_2)^{[L':L]}/\eta(w'_1,w'_2)$$
and
$$\vartheta(z;L,L')=C(\omega,\omega')\theta(z;\omega)^{[L':L]}/\theta(z;\omega')=\delta(L,L')\prod_{u\in T}(\wp(z;L)-\wp(u;L))^{-1}$$
only depend on the lattices $L,L'$ and so that $\vartheta$ satisfies the distribution relation
\begin{equation}\label{equ:4.1}
  \vartheta(z;M,M')=\prod_{i=1}^{[M:L]}\vartheta(z+t_i;L,L')
\end{equation}
for any lattice $L\subseteq M$ such that $M\cap L'=L$ (and where $M'=M+L'$), where $t_i\in M$ are a set of representatives of $M/L$. The set $T$ is any set of representatives of $(L'\setminus \{0\})/(\pm 1\ltimes L)$ and $\wp$ is the Weierstrass $\wp$-function associated to $L$. In particular $\vartheta(z;L,L')$ is an elliptic function.

Kato reproves Robert's result in a scheme theoretic context, which is demonstrated as follows.

\begin{lemma} (Kato\cite{Kat04}(15.4.))\
  Let $E$ be an elliptic curve over a field $F$ with $\mathcal{O}_K\cong End_F(E)$ and $\mathfrak{a}$ an ideal in $\mathcal{O}_K$ prime to $6$. Then there
  is a unique function
  $$\Theta_\mathfrak{a}\in \Gamma(E\setminus \ker([\mathfrak{a}]), \mathcal{O}^\times)$$
  with

  (i) $div(\Theta_\mathfrak{a})=N\mathfrak{a}\cdot(0)-\ker([\mathfrak{a}])$

  (ii) For any $m\in \mathbb{Z}$ prime to $\mathfrak{a}$ we have $N_m(\Theta_\mathfrak{a})=\Theta_\mathfrak{a}$, where $N_m$ is the norm map associated to the finite flat morphism $E\setminus \ker([m\mathfrak{a}])\to E\setminus \ker([\mathfrak{a}])$ given by multiplication with $m$.

 For $F=\mathbb{C}$ we have
  $$\Theta_\mathfrak{a}(z)=\vartheta(z;L,\mathfrak{a}^{-1}L).$$
\end{lemma}

Given $\mathfrak{f}\neq1$ and any $\mathfrak{a}$ which is prime to $6\mathfrak{f}$, the elliptic unit is defined by
$${}_\mathfrak{a}z_\mathfrak{f}=\vartheta(1;\mathfrak{f},\mathfrak{a}^{-1}\mathfrak{f})$$
and for $\mathfrak{f}=1$ we define a family of elements indexed by all ideals $\mathfrak{a}$ of $K$ by
$$u(\mathfrak{a})=\frac{\Delta(\mathcal{O}_K)}{\Delta(\mathfrak{a}^{-1})}.$$

\begin{lemma}
  The complex numbers ${}_\mathfrak{a}z_\mathfrak{f}$ and $u(\mathfrak{a})$ satisfy the following properties

  (i) (Rationality) ${}_\mathfrak{a}z_\mathfrak{f}\in K(\mathfrak{f}),\ u(\mathfrak{a})\in K(1)$;

  (ii) (Integrality)
  \begin{align*}
     &{}_\mathfrak{a}z_\mathfrak{f}\in
  \begin{cases}
    \mathcal{O}_{K(\mathfrak{f})}^\times, & \mathfrak{f}\ \mbox{divisible by primes }\ \mathfrak{p}\neq\mathfrak{q}  \\
    \mathcal{O}_{K(\mathfrak{f}),\{v|\mathfrak{f}\}}^\times, & \mathfrak{f=p}^n\ \mbox{for some prime}\ \mathfrak{p}.
  \end{cases}  \\
     & u(\mathfrak{a})\cdot\mathcal{O}_{K(1)}=\mathfrak{a}^{-12}\mathcal{O}_{K(1)}
  \end{align*}

  (iii) (Galois action) For $(\mathfrak{c,fa})=1$ with Artin symbol $\sigma_\mathfrak{c}\in Gal(K(\mathfrak{f})/K)$ we have
  $${}_\mathfrak{a}z_\mathfrak{f}^{\sigma_\mathfrak{c}}=\vartheta(1;\mathfrak{c}^{-1}\mathfrak{f},\mathfrak{c}^{-1}\mathfrak{a}^{-1}\mathfrak{f}); \quad
  u(\mathfrak{a})^{\sigma_\mathfrak{c}}=u(\mathfrak{ac})/u(\mathfrak{c}).$$
  This implies
  $${}_\mathfrak{a}z_\mathfrak{f}^{N\mathfrak{c}-\sigma_\mathfrak{c}}={}_\mathfrak{c}z_\mathfrak{f}^{N\mathfrak{a}-\sigma_\mathfrak{a}};
  \quad  u(\mathfrak{a})^{1-\sigma_\mathfrak{c}}= u(\mathfrak{c})^{1-\sigma_\mathfrak{a}}.$$
  (iv) (Norm compatibility) For a prime ideal $\mathfrak{p}$ we have
  \[
  N_{K(\mathfrak{pf})/K(\mathfrak{f})}({}_\mathfrak{a}z_\mathfrak{pf})^{w_\mathfrak{f}/w_{\mathfrak{pf}}}=
  \begin{cases}
    {}_\mathfrak{a}z_\mathfrak{f}, & \mbox{if } \ \mathfrak{p}\mid\mathfrak{f}\neq 1 \\
    {}_\mathfrak{a}z_\mathfrak{f}^{1-\sigma_\mathfrak{p}^{-1}}, & \mbox{if }\ \mathfrak{p}\nmid\mathfrak{f}\neq 1 \\
    u(\mathfrak{p})^{(\sigma_\mathfrak{a}-N\mathfrak{a})/12}, & \mbox{if}\ \ \mathfrak{f}=1.
  \end{cases}
  \]
\end{lemma}

In order to compute the image of $\xi_\mathfrak{m}(j)$ under the \'{e}tale Chern class map $\rho_{et}$, one need the following commutative diagram

\[\begin{tikzcd}
		H_{\mathcal{M}}^{2j+1}(E^{2j}, 2j+1) \arrow[r, "\rho_{et}"] \arrow[d, "\mathcal{K}_\mathcal{M}"'] & H^{2m+1}(E^{2j}, \mathbb{Q}_\ell(2j+1))  \arrow[d, "\mathcal{K}_\ell"] \\
		H^1_{\mathcal{M}}(\text{Spec}(K(\mathfrak{f})),j+1) \arrow[r, "\rho_{et}"'] & H^1(K(\mathfrak{f}),\mathbb{Q}_\ell(j+1))
	\end{tikzcd} \]

By \cite[Theorem2.2.4]{HK99}, the \'{e}tale realization of the Eisenstein symbol can be computed in term of the pullback of the elliptic polylogarithm
along torsion section. Then,
$$\rho_{et}(\xi_\mathfrak{m}(j))=\mathcal{K}_\ell(\rho_{et}(\mathcal{E}is^{2j}(\rho_\mathfrak{m}\cdot\beta)))=
\psi(\rho_\mathfrak{m})^{-1}\cdot \mathcal{N}\mathfrak{m}^{2j-1}\mathcal{K}_\ell(\beta^*\mathcal{P}ol_{\mathbb{Q}_\ell})^{2j}.$$
Kings \cite[Thm.4.2.9]{Kin01} gets the explicit $\ell$-adic Eisenstein class for $\ell \nmid \mathfrak{f}$, i.e.
$$(\beta^*\mathcal{P}ol_{\mathbb{Q}_\ell})^{2j}=\pm\frac{1}{\mathcal{N}\mathfrak{a}([\mathfrak{a}]^{2j}\mathcal{N}\mathfrak{a}-1)(2j)!}
\left(\delta\sum_{[\ell^n]t_n=\beta}\Theta_\mathfrak{a}(-t_n)\tilde{t}_n^{\otimes 2j}\right)_n$$
where $\delta$ is the connecting homomorphism in Kummer sequence, the $\Theta_\mathfrak{a}(-t_n)$ are elliptic units, $\tilde{t}_n$ is the projection of $t_n$ to $E[\ell^n]$, and $\mathfrak{a}\subset\mathcal{O}_K$ is chosen prime to $\ell\mathfrak{f}$.

Consider the elliptic curve $E$ over $K(\mathfrak{f})$ with a uniformization $\mathbb{C}/\Gamma$. We can define a multiplication by $\rho_\mathfrak{m}$
on the elliptic curve componentwise. By the choice of torsion point $\beta$, we see $\rho_\mathfrak{m}t_n\in E[\ell^n]$. Taking this as our $\tilde{t}_n$ we have to multiply King's result by a factor of $\rho_\mathfrak{m}^{-2j}$.

For $t\in E[\ell^n]$, define $\gamma(t)^k:=\langle t,\sqrt{d_K}t\rangle^{\otimes k}$, where $\langle,\ \rangle$ is the Weil pairing. Then the Kronecker map acts on the Tate module by
$$\mathcal{K}_\ell(\tilde{t}_n^{\otimes 2j})=\gamma(\tilde{t}_n)^j=\zeta_{\ell^n}^{\otimes j}, $$
on the integral isogenies $[\mathfrak{a}]$ by
$$\mathcal{K}_\ell([\mathfrak{a}]^{2j})=\mathcal{N}\mathfrak{a}^j$$
and on the idele $\rho_\mathfrak{m}$ by $\mathcal{K}_\ell(\rho_\mathfrak{m}^{-2j})=\mathcal{N}(\rho_\mathfrak{m})^{-j}=\mathcal{N}\mathfrak{m}^{-j}.$
The Artin automorphism $\sigma_\mathfrak{a}$ acts on the space $H^1(K(\mathfrak{f}),\mathbb{Q}_\ell)$ by $\mathcal{N}\mathfrak{a}$ and thus on the space
$H^1(K(\mathfrak{f}),\mathbb{Q}_\ell(j+1))$ by $\mathcal{N}\mathfrak{a}^{j+2}$. At last, Johnson concludes that (see \cite[p.34]{J})
  $$\rho_{et}(\xi_\mathfrak{m}(j))=\frac{\psi(\rho_\mathfrak{m})^{-1}}{\mathcal{N}\mathfrak{a}-\sigma_\mathfrak{a}}\cdot\mathcal{N}\mathfrak{m}^{j-1}
\left(\delta\sum_{[\ell^n]t_n=\beta}\Theta_\mathfrak{a}(-t_n)\zeta_{\ell^n}^{\otimes j}\right)_n.$$
Since there is(see \cite[Lemma3.2.4]{J})
\[ \prod_{\mathfrak{l}|\ell}(1-\mbox{Fr}_\mathfrak{l}^{-1})^{-1}\left(\sum_{[\ell^n]t_n=\Omega f_{\mathfrak{m}}^{-1}}\Theta_\mathfrak{a}(-t_n)\zeta_{\ell^n}^{\otimes j}\right)_n=w_{\mathfrak{m}}(\mbox{Tr}_{K(\ell^n\mathfrak{m})/K(\mathfrak{m})}\Theta_\mathfrak{a}(-s_n)\zeta_{\ell^n}^{\otimes j})_n, \]
where $s_n$ is a primitive $\ell^n$th root of $\beta$, we obtain the following theorem (see \cite[Theorem3.2.1]{J})
\begin{theorem}\label{thm:l-adic-regulator}
  For all $1\neq \mathfrak{m}|\mathfrak{f}$, we have that
  \[ \rho_{et}(\xi_\mathfrak{m}(j))=\frac{\mathcal{N}\mathfrak{m}^{j-1}w_\mathfrak{m}}{(\mathcal{N}\mathfrak{a}-\sigma_\mathfrak{a})
  \prod_{\mathfrak{l}|\ell}(1-\mbox{Fr}_\mathfrak{l}^{-1})(2j)!}
  (\mbox{Tr}_{K(\ell^n\mathfrak{m})/K(\mathfrak{m})}\Theta_\mathfrak{a}(-s_n)\zeta_{\ell^n}^{\otimes j})_n
  \]
\end{theorem}

\section{``Main conjecture'' of Iwasawa theory for quadratic imaginary fields}

Fix an elliptic curve $E$ defined over imaginary quadratic field $K$,  with $\text{End}_K(E)=\mathcal{O}_K$. Fix a prime $\mathfrak{p}$ of $K$ where $E$
has good reduction. Write $K_n=K(E[\mathfrak{p}^n]), n=0,1,2,\cdots, \infty$. We have
$$\text{Gal}(K_\infty/K)\cong\mathcal{O}_\mathfrak{p}^\times\cong\text{Gal}(K_\infty/K_1)\times\text{Gal}(K_1/K)$$
where
$$\text{Gal}(K_1/K)\cong(\mathcal{O}/\mathfrak{p})^\times$$
is the prime-to-$p$ part of $\text{Gal}(K_\infty/K)$ and
$$\text{Gal}(K_\infty/K_1)\cong 1+\mathfrak{p}\mathcal{O}_\mathfrak{p}\cong\mathbb{Z}_p^{[K_\mathfrak{p}:\mathbb{Q}_p]}$$
is the $p$-part.

 Let $F=K(\mathfrak{f}), (\mathfrak{f,p})=1,$ and $F_n=K(\mathfrak{fp}^n)\supset F(\mu_{p^n}), n=0,\cdots,\infty$. Write $\mathscr{G}_F=\text{Gal}(F_\infty/F)\subset \mathscr{G}_\infty=\text{Gal}(F_\infty/K), \Delta=\text{Gal}(F_1/K),$  $ \Gamma=\text{Gal}(F_\infty/F_1)$.
Let $\mathscr{A}_n=p$-part of the class group of $F_n$,  $\mathscr{A}'_n=p$-part of the $S$-class group of $F_n$, $\mathscr{A}_\infty=\varprojlim \mathscr{A}_n, \mathscr{A}'_\infty=\varprojlim \mathscr{A}'_n$. Put $\mathscr{X}_n=G_S^{\text{ab}}(K_n)\otimes \mathbb{Z}_p$, $\mathscr{X}_\infty=\varprojlim \mathscr{X}_n.$ Write $U'(F_n)=$group of $S$-units of $F_n$, $\overline{U}'_n=U'_n\otimes\mathbb{Z}_p,$ $\overline{U}'_\infty=\varprojlim \overline{U}'_n, \tilde{U}'_n=\text{fr}_{\mathbb{Z}_p}\overline{U}'_n, \tilde{U}'_\infty=\text{fr}_{\Lambda}\overline{U}'_\infty$.

By \cite[Corollary5.20]{R1}, we have
\begin{equation}\label{equ:6-1}
  \Gamma=\text{Gal}(F_\infty/F_1)=\varprojlim \text{Gal}(F_n/F_1)\cong \varprojlim \text{Gal}(K_n/K_1).
\end{equation}

For prime $\mathfrak{p}$ of $K$ over $p$, suppose $p\nmid [F_1:K]$ and $p$ does not divide the number of roots of unity in the Hilbert class field of $K$. From (\ref{equ:6-1}), we have $ \Gamma=\text{Gal}(F_\infty/F_1)\cong \mathbb{Z}_p$ or $\mathbb{Z}_p^2$.  Then $\mathscr{G}_\infty=\text{Gal}(F_\infty/K)\cong\Delta\times\Gamma$. Write $\mathscr{A}(F_n)$ for the $p$-part of the ideal class group of $F_n$, $\mathscr{E}(F_n)$ for the group of global units of $F_n$, and $\mathscr{C}(F_n)$ for the group of Robert elliptic units. Let $\mathscr{U}(F_n)$ be the group of local units of $F_n\otimes_K K_\mathfrak{p}$ which are congruent to 1 modulo the primes above $\mathfrak{p}$. Let $\overline{\mathscr{E}}(F_n)$ and $\overline{\mathscr{C}}(F_n)$ denote the closure of $\mathscr{E}(F_n)\cap \mathscr{U}(F_n)$ and $\mathscr{C}(F_n)\cap \mathscr{U}(F_n)$ in
$\mathscr{U}(F_n)$. For $F_\infty$, we define
\begin{align*}
   & \mathscr{A}(F_\infty)=\varprojlim \mathscr{A}(F_n), \ \ \overline{\mathscr{E}}(F_\infty)=\varprojlim \overline{\mathscr{E}}(F_n),\\
   & \overline{\mathscr{C}}(F_\infty)=\varprojlim \overline{\mathscr{C}}(F_n),\  \ \mathscr{U}(F_\infty)=\varprojlim \mathscr{U}(F_n),
\end{align*}
with respect to norm maps.    Let $M(F_n)$ for the maximal abelian $p$-extension of $F_n$ which is unramified outside the primes above $\mathfrak{p}$, and write $\mathscr{X}(F_n)=\text{Gal}(M(F_n)/F_n)$. We will also write $\mathscr{A}_\infty,\ \mathscr{U}_\infty, \overline{\mathscr{E}}_\infty,\ \overline{\mathscr{C}}_\infty$ and
$\mathscr{X}_\infty$ for $\mathscr{A}(F_\infty),\ \mathscr{U}(F_\infty), \overline{\mathscr{E}}(F_\infty),\ \overline{\mathscr{C}}(F_\infty)$ and $\mathscr{X}(F_\infty)$, respectively. Then global class field theory gives us the following exact sequence(The details can be found in\cite{dS}Chapter3 \S1.7)
\begin{equation}\label{equ:exact sequ}
  0\to \overline{\mathscr{E}}_\infty/\overline{\mathscr{C}}_\infty \to \mathscr{U}_\infty/ \overline{\mathscr{C}}_\infty\to \mathscr{X}_\infty \to \mathscr{A}_\infty \to 0.
\end{equation}

If $Y$ is a $\mathbb{Z}_p[\Delta]$-module and $\chi$ is an irreducible $\mathbb{Z}_p$-representation of $\Delta$, define $Y^\chi$ to be the $\chi$-isotypic component of $Y$ and
$$e_\chi=\frac{1}{\#\Delta}\sum_{\tau\in\Delta}\text{Tr}(\chi(\tau))\tau^{-1}\in \mathbb{Z}_p[\Delta].$$
Then $Y^\chi=e_\chi Y$.

Define the Iwasawa algebra
$$\Lambda=\mathbb{Z}_p[[\mathscr{G}_\infty]]=\varprojlim \mathbb{Z}_p[\text{Gal}(F_n/K)].$$
Then for every irreducible $\mathbb{Z}_p$-representation $\chi$ of $\Delta$,
$$\Lambda^\chi= \mathbb{Z}_p[[\mathscr{G}_\infty]]^\chi=R_\chi[[\Gamma]],$$
where $R_\chi$ is the ring of integers of the unramified extension of $\mathbb{Q}_p$ of degree dim($\chi$). Further, $R_\chi[[\Gamma]]$ is (noncanonically) isomorphic to a power series rings in 1 or 2 variables over $R_\chi$ according to $\Gamma\cong \mathbb{Z}_p$ or $\mathbb{Z}_p^2.$

We have $\Lambda=\bigoplus_\chi \Lambda^\chi.$ We will call a $\Lambda$-module $Y$ a torsion if every element of $Y$ is annihilated by an element of $\Lambda$ which is  not a zero-divisor, or equivalently if and only if $Y^\chi$ is a torsion $\Lambda^\chi$-module for every $\chi$.

A finitely generated $\Lambda$-module is called pseudo-null if it is annihilated by an ideal of height 2; if $\Gamma\cong\mathbb{Z}_p$ the pseudo-null
modules are exactly the finite modules. A pseudo-isomorphism of $\Lambda$-modules is a map with pseudo-null kernel and cokernel. It follows from the well known classification theorem for $\Lambda$-modules that for every finitely generated torsion $\Lambda$-module $Y$ can find elements $f_i$ of
$\Lambda$ and pseudo-isomorphisms
$$Y\to \oplus \Lambda/f_i\Lambda \quad \text{and} \quad \oplus \Lambda/f_i\Lambda \to Y.$$
The characteristic ideal $(\Pi f_i)/\Lambda$ is a well-defined invariant of $Y$ which we will denote by char$(Y)$. For every $\chi$, char$(Y)^\chi\subset \Lambda^\chi$ is the usual character ideal of the $\Lambda^\chi$-module $Y^\chi$, and a generator  of char$(Y)^\chi$ is usually called a characteristic power series of $Y^\chi$.

Rubin proves that $\mathscr{A}_\infty,\ \mathscr{U}_\infty, \overline{\mathscr{E}}_\infty,\ \overline{\mathscr{C}}_\infty$ and $\mathscr{X}_\infty$ are all finitely generated $\Lambda$-modules, and $\mathscr{A}_\infty$ and $\overline{\mathscr{E}}_\infty/\overline{\mathscr{C}}_\infty$ are torsion $\Lambda$-modules in \cite{R}. Further,
he got the following main conjecture of Iwasawa theory for imaginary quadratic fields, that is,
\begin{theorem}\label{thm:Imagi-quad-conj-of-Iwasa}
 (i) Suppose $p$ splits into two distinct primes in $K$. Then
  $$ char(\mathscr{A}_\infty)=char(\overline{\mathscr{E}}_\infty/\overline{\mathscr{C}}_\infty) \quad \text{and} \quad
  char(\mathscr{X}_\infty)=char(\mathscr{U}_\infty/\overline{\mathscr{C}}_\infty)$$

 (ii) Suppose $p$ remains prime or ramifies in $K$. Then
 $$  char(\mathscr{A}_\infty)\quad \text{divides} \quad char(\overline{\mathscr{E}}_\infty/\overline{\mathscr{C}}_\infty).$$
 If $\chi$ is an irreducible $\mathbb{Z}_p$-representation of $\Delta$ which is nontrivial on the decomposition group of $\mathfrak{p}$ in $\Delta$,
 then
 $$char(\mathscr{A}_\infty)^\chi=char(\overline{\mathscr{E}}_\infty/\overline{\mathscr{C}}_\infty)^\chi.$$
\end{theorem}

\section{\'{e}tale cohomology and Iwasawa module}

For an odd prime number $p$, let $S=S(F)$ denote the set of primes of $F$ above $p$. Denote the maximal algebraic $S$-ramified extension of $F$ by $\Omega_S$ and set $G_S(F)=\text{Gal}(\Omega_S/F)$. It is well known that the $p$-adic cohomology group $H^i(G_S(F),\mathbb{Z}_p(m))$$=\varprojlim H^i(G_S(F),\mathbb{Z}/p^n\mathbb{Z}(m)),\ m\in\mathbb{Z}$, coincide with the \'{e}tale cohomology groups $H^i_{et}(\text{Spec}\mathcal{O}_F[1/p],\mathbb{Z}_p(m))$. Put $H^1(\mathcal{O}_F[1/p],\mathbb{Q}_p(m)):=H^1_{et}(\text{Spec}\mathcal{O}_F[1/p],\mathbb{Q}_p(m))$.

There is a unified description of the groups $H^1(\mathcal{O}_F[1/p],\mathbb{Q}_p(m))$ in terms of Iwasawa modules which goes back to Deligne\cite{Del} and Soul\'{e}\cite{Sou1, Sou2}. Let $E_n=F(\mu_{p^n}),\ E_\infty=F(\mu_{p^\infty}),\ G_\infty=\text{Gal}(E_\infty/F),\ U'_n=U'_{E_n},$ and $\overline{U}'_\infty=\varprojlim U'_n\otimes \mathbb{Z}_p.$ In \cite{Sou1,Sou2}, Soul\'{e} proved that $(\overline{U}'_\infty(m-1))_{G_\infty}\otimes\mathbb{Q}_p\cong K_{2m-1}(F)\otimes\mathbb{Q}_p\cong H^1(\mathcal{O}_F[1/p],\mathbb{Q}_p(m))$ for $m\geq2.$ Kolster et.al. (see \cite{KNF}) extended Soul\'{e}'s results for all $m\in \mathbb{Z}, m\neq1$, modulo the standard conjecture.

The descent modules $\overline{U}'_\infty(m-1)_{G_\infty}$ is very useful. It has been shown by Kolster et.al.  for $m\neq0,1$ and by Kuz'min\cite{Ku} for $m=1$ that $\overline{U}'_\infty(m-1)_{G_\infty}$ injects into $H^1(\mathcal{O}_F[1/p],\mathbb{Z}_p(m))$. Assume that $F$ is abelian field.  Consider the projective limit of Sinnott's circular $S$-units $C'_n\subset U'_n$. Let $\overline{C}'_\infty=\varprojlim C'_n\otimes\mathbb{Z}_p.$ Again, Kolster et.al. proved $\overline{C}'_\infty(m-1)_{G_\infty}$ injects into $H^1(\mathcal{O}_F[1/p],\mathbb{Z}_p(m))$ with finite index. If $F$ is an abelian extension of imaginary quadratic field, we use the elliptic units to get a similar result as them.

One essential tool used is Poitou-Tate's long exact sequence, which gives a description of the kernels and cokernels of localization maps in Galois or \'{e}tale cohomology. Let us recall the main facts to be used. Denote by \font\fontWCA=wncyr10 {\fontWCA SH}$_S^i(F,\mathbb{Q}_p/\mathbb{Z}_p(m))$ and
\font\fontWCA=wncyr10 {\fontWCA SH}$_S^i(F,\mathbb{Z}_p(m))$ respectively the kernels of the localization maps
\[ H^i(\mathcal{O}_F[1/p],\mathbb{Q}_p/\mathbb{Z}_p(m))\to \bigoplus_{v\in S}H^i(F_v,\mathbb{Q}_p/\mathbb{Z}_p(m)) \]
and
\[ H^i(\mathcal{O}_F[1/p],\mathbb{Z}_p(m))\to \bigoplus_{v\in S}H^i(F_v,\mathbb{Z}_p(m)). \]
By Poitou-Tate, we have canonical isomorphisms(for any $m\in \mathbb{Z}$)
\[ H^i(F_v,\mathbb{Z}_p(m))\cong H^j(F_v,\mathbb{Q}_p/\mathbb{Z}_p(1-m))^* \quad (0\leq i,j\leq2,\ i+j=2) \]
and
\[ \text{\font\fontWCA=wncyr10 {\fontWCA SH}}_S^i(F,\mathbb{Z}_p(m)) \cong \text{\font\fontWCA=wncyr10 {\fontWCA SH}}_S^i(F,\mathbb{Q}_p/\mathbb{Z}_p(1-m))^* \quad (0\leq i,j\leq2,\ i+j=3),\]
where $(\cdot)^*$ denotes the Pontryagin dual.

We will use Poitou-Tate's exact sequence in the following form
\begin{multline*}
  0\to \text{\font\fontWCA=wncyr10 {\fontWCA SH}}_S^2(F,\mathbb{Q}_p/\mathbb{Z}_p(1-m))^* \to H^1(\mathcal{O}_F[1/p],\mathbb{Z}_p(m))\to \bigoplus_{v\in S}H^1(F_v,\mathbb{Z}_p(m)) \\
  \cong \bigoplus_{v\in S}H^1(F_v,\mathbb{Q}_p/\mathbb{Z}_p(1-m))^* \to H^1(\mathcal{O}_F[1/p],\mathbb{Q}_p/\mathbb{Z}_p(1-m))^* \to \text{\font\fontWCA=wncyr10 {\fontWCA SH}}_S^2(F,\mathbb{Z}_p(m)) \to 0.
\end{multline*}

Note that for $m=1$, this is essentially the class field theory exact sequence. It is known that $H^2(F_v,\mathbb{Q}_p/\mathbb{Z}_p(m))=0 $ for all $m\neq1$, so
\[ \text{\font\fontWCA=wncyr10 {\fontWCA SH}}_S^2(F,\mathbb{Q}_p/\mathbb{Z}_p(m))=H^2(\mathcal{O}_F[1/p],\mathbb{Q}_p/\mathbb{Z}_p(m)) \]
for all $m\neq1$. If $m\geq2$, then $H^2(\mathcal{O}_F[1/p],\mathbb{Q}_p/\mathbb{Z}_p(m))=0$, which is essentially equivalent to the finiteness of $K_{2m-2}(\mathcal{O}_F).$

We will use Iwasawa theory to describe the groups $H^1(\mathcal{O}_F[1/p],\mathbb{Z}_p(m))$ as modules of descent. In the process of descent, we will use the following generalized Tate's lemma.

\begin{lemma}\label{lem:Tate}
  ( Tate's lemma.) \  For any $m\neq0$, $H^1(\mathscr{G}_F, \mathbb{Q}_p/\mathbb{Z}_p(m))=0.$ \textcolor{red}{\textbf{???}}
\end{lemma}
\begin{proof}
   If $\mathscr{G}_F\cong\mathbb{Z}_p\times \Delta$, the lemma is the case \cite[Lemma2.1]{KNF}.

   Suppose $\mathscr{G}_F\cong\mathbb{Z}_p^2\times \Delta$.  Let $L=F(\mu_{p^\infty})$. Then $\text{Gal}(L/F)\cong\mathbb{Z}_p\times\Delta'$  and $F_\infty/L$ is $\mathbb{Z}_p$-extension. Write $H=\text{Gal}(F_\infty/L)$.
  Hence there are
  $$H^1(\text{Gal}(L/F), (\mathbb{Q}_p/\mathbb{Z}_p(m))^H)=H^1(\text{Gal}(L/F), \mathbb{Q}_p/\mathbb{Z}_p(m))=0$$
   and
  $H^1(H, \mathbb{Q}_p/\mathbb{Z}_p(m))=0.$
  The cohomology exact sequence
  $$0\to H^1(\text{Gal}(L/F), (\mathbb{Q}_p/\mathbb{Z}_p(m))^H)\xrightarrow{Inf} H^1(\mathscr{G}_F, \mathbb{Q}_p/\mathbb{Z}_p(m))\xrightarrow{Res}
  H^1(H, \mathbb{Q}_p/\mathbb{Z}_p(m))$$
  shows that $H^1(\mathscr{G}_F, \mathbb{Q}_p/\mathbb{Z}_p(m))=0.$
\end{proof}

\begin{theorem}\label{thm:etale-iwasawa}
  For $m\neq1$, we have canonical isomorphisms \textcolor{red}{\textbf{???}}
  $$ H^1(\mathcal{O}_{F}[1/p],\mathbb{Q}_p/\mathbb{Z}_p(1-m))^*\cong \mathscr{X}_\infty(m-1)_{\mathscr{G}_F} , $$
  $$ H^2(\mathcal{O}_{F}[1/p],\mathbb{Q}_p/\mathbb{Z}_p(1-m))^*
\cong \mathscr{X}_\infty(m-1)^{\mathscr{G}_F}$$
\end{theorem}
\begin{proof}
  The proof is inspired by Coates's in \cite{C}. Let $F_S$ for the maximal extension of $F$ which is unramified outside of $S=\{p, \infty\}$. Let $G_S(F_\infty)$ be the subgroup of $G_S$ fixing $F_\infty$. Both $\mathscr{G}_F$ and $G_S(F_\infty)$ act on $\mathbb{Q}_p/\mathbb{Z}_p(1-m)$.
  Also $\mathscr{G}_F$ has cd$_p \mathscr{G}_F=1$. Hence the Hochschild-Serre spectral sequence
  $$E^{i,j}_2=H^i(\mathscr{G}_F, H^j(G_S(F_\infty),\mathbb{Q}_p/\mathbb{Z}_p(1-m)))\Rightarrow H^{i+j}(G_S, \mathbb{Q}_p/\mathbb{Z}_p(1-m))$$
  degenerates into the two short exact sequences
  \begin{align*}
    0 & \to  H^1(\mathscr{G}_F, H^{i-1}(G_S(F_\infty),\mathbb{Q}_p/\mathbb{Z}_p(1-m)))\to H^{i}(G_S, \mathbb{Q}_p/\mathbb{Z}_p(1-m))\\
     & \to H^{i}(G_S(F_\infty),\mathbb{Q}_p/\mathbb{Z}_p(1-m))^{\mathscr{G}_F}\to 0
  \end{align*}
  with $i=1,2.$  Since $G_S(F_\infty)$ acts trivially on $\mathbb{Q}_p/\mathbb{Z}_p(1-m)$ , we have
  $$H^{1}(G_S(F_\infty),\mathbb{Q}_p/\mathbb{Z}_p(1-m))=\text{Hom}(\mathscr{X}_\infty,\mathbb{Q}_p/\mathbb{Z}_p(1-m))
  =(\mathscr{X}_\infty(m-1))^*.$$
  Moreover, for $m\neq1$, we have $ H^1(\mathscr{G}_F,\mathbb{Q}_p/\mathbb{Z}_p(1-m))=0$ by Lemma\ref{lem:Tate}. Finally, there is
  $$H^2(G_S(F_\infty),\mathbb{Q}_p/\mathbb{Z}_p(1-m))=H^2(G_S(F_\infty),\mathbb{Q}_p/\mathbb{Z}_p)(1-m)=0 \quad\textcolor{red}{\textbf{???}}$$
  The assertions of the theorem now follow from these remarks and the above two exact sequences. This completes the proof.
\end{proof}

\begin{theorem}
  For $m\neq0,1$, we have a canonical exact sequence
  $$0\to (\overline{U}'_\infty(m-1))_{\mathscr{G}_F}\to H^1(\mathcal{O}_F[1/p],\mathbb{Z}_p(m))\to \mathscr{A}'_\infty(m-1)^{\mathscr{G}_F}
  \to 0.$$
\end{theorem}
\begin{proof}
  The  exact sequence
  $$0\to \mathbb{Z}/p^n\mathbb{Z}(1)\to F_n^\times\xrightarrow{p^n}F_n^\times \to 0$$
  leads to injective map
  $$\varprojlim_n\overline{U}'_n \to \varprojlim_n H^1(\mathcal{O}_{F_n}[1/p],\mathbb{Z}_p(1)).$$
  Twisting the above map by Tate module $\mathbb{Z}_p(m-1)$ and taking coinvariants with respect to $\mathscr{G}_F$, we have the injective homomorphism
  $$0\to (\overline{U}'_\infty(m-1))_{\mathscr{G}_F}\to H^1(\mathcal{O}_F[1/p],\mathbb{Z}_p(m)).$$
\end{proof}

We define a map in the sprit of Soul\'{e}
\[ e_p: \overline{\mathscr{C}}_\infty(m-1) \to H^1(\mathcal{O}_F[1/p],\mathbb{Z}_p(m)).\]
Write
\[ H^1(\mathcal{O}_F[1/p],\mathbb{Z}_p(m))=\varprojlim_r H^1(\mathcal{O}_F[1/p],\mu_{p^r}^{\otimes m}).\]
 For a norm compatible system of elliptic units $(\theta_n)_n$ and an element $(t_n)_n\in \mu_{p^n}^{\otimes m-1}$, put
\[ e_p((\theta_r\otimes t_{n})_n) :=(N_{F_n/F}(\theta_n\otimes t_{n}))_n \]
where $\theta_n\otimes t_{n}$ is an element in
\[ \mathcal{O}_{F_n}[1/p]^*/(\mathcal{O}_{F_n}[1/p]^*)^{p^n}\otimes \mu_{p^n}^{\otimes m-1}\subset H^1(\mathcal{O}_{F_n}[1/p],\mu_{p^n}^{\otimes m}),  \]
where the inclusion comes from Kummer theory and $N_{F_n/F}$ is the norm map on the cohomology. The map $e_p$ factors through the covariants under $\mathscr{G}_F$, so we have the following map
\begin{theorem}
  Let $F$ be an abelian extension over $K$ and $m>1$. Then the map
  $$e_p: \overline{\mathscr{C}}_\infty(m-1)_{\mathscr{G}_F}\to H^1(\mathcal{O}_F[1/p],\mathbb{Z}_p(m))$$
  is injective.
\end{theorem}
\begin{proof}
We first show that the finiteness of $H^2(\mathcal{O}_F[1/p],\mathbb{Z}_p(m))$ implies the finiteness of $(\mathscr{A}_\infty(m-1))_{\mathscr{G}_F}$.
Using \cite[Lemma6.2]{R} we know that this implies that $(\mathscr{A}_\infty(m-1))^{\mathscr{G}_F}$ is finite. It suffices to show that the kernel of
$e_p$ on $\overline{\mathscr{E}}_\infty(m-1)$ is finite by \cite[Corollary7.8]{R} because both $\overline{\mathscr{E}}$ and $\overline{\mathscr{C}}$
are $\Lambda$-modules of rank 1 with $\overline{\mathscr{E}}/\overline{\mathscr{C}}$ a torsion module. Suppose that the image of $(\mathscr{E}_\infty(m-1))_{\mathscr{G}_F}$ under $e_p$ has not rank 1, i.e. is finite. Then the image of $(\overline{\mathscr{E}}_\infty(m-1))_{\mathscr{G}_F}$
in $(\mathscr{U}_\infty(m-1))_{\mathscr{G}_F}$ must be finite as well because of
 \[(\mathscr{U}_\infty(m-1))_{\mathscr{G}_F}\cong H^1(F\otimes \mathbb{Q}_p, \mathbb{Q}_p/\mathbb{Z}_p(1-m))^*:=\left(\bigoplus_{v|p}H^1(F_v,\mathbb{Q}_p/\mathbb{Z}_p(1-m))\right)^*,\]
  which is the local case of Theorem\ref{thm:etale-iwasawa}\textcolor{red}{???}. The kernel of the map
\[ (\overline{\mathscr{E}}_\infty(m-1))_{\mathscr{G}_F} \to (\mathscr{U}_\infty(m-1))_{\mathscr{G}_F}\]
is $H_1(\mathscr{G}_F, \mathscr{U}_\infty/ \overline{\mathscr{E}}_\infty(m-1))$. On the other hand, up to finite groups, there exists $H_1(\mathscr{G}_F, \mathscr{X}_\infty(m-1))\cong H^2(\mathcal{O}_F[1/p], \mathbb{Q}_p/\mathbb{Z}_p(1-m))^*$\textcolor{red}{???}. So, we get a commutative
diagram (up to finite groups)
\[
   \begin{tikzcd}
   H_1(\mathscr{G}_F, \mathscr{U}_\infty/ \overline{\mathscr{E}}_\infty(m-1)) \arrow[d]\arrow[r, "\iota"] & H_1(\mathscr{G}_F, \mathscr{X}_\infty(m-1))\arrow[d] \\
   (\overline{\mathscr{E}}_\infty(m-1))_{\mathscr{G}_F} \arrow[d]\arrow[r, "e_p"] &  H^1(\mathcal{O}_F[1/p], \mathbb{Z}_p(m)) \arrow[d]\\
   (\mathscr{U}_\infty(m-1))_{\mathscr{G}_F} \arrow[r, "\cong"] & H^1(F\otimes \mathbb{Q}_p,\mathbb{Q}_p/\mathbb{Z}_p(1-m)))^*.
   \end{tikzcd}  \]
 The kernel of the map $\iota$ is a quotient of $(\mathscr{A}_\infty(m-1))^{\mathscr{G}_F}$ which is finite. Thus, we get a contradiction and $e_p$
 can not be zero on the free part of $(\overline{\mathscr{E}}_\infty(m-1))_{\mathscr{G}_F}$. Hence, $e_p$ is non zero on $(\overline{\mathscr{C}}_\infty(m-1))_{\mathscr{G}_F}$.
\end{proof}

Consider the exact sequence
\begin{equation*}
  0\to \overline{\mathscr{E}}_\infty/\overline{\mathscr{C}}_\infty \to \mathscr{U}_\infty/ \overline{\mathscr{C}}_\infty\to \mathscr{X}_\infty \to \mathscr{A}_\infty \to 0.
\end{equation*}
Let $\chi$ be a character of $G=\text{Gal}(F/K)$ (viewed also as a character of $\Delta=\text{Gal}(F_1/K)$). From the above exact sequence, we obtain an exact sequence of $\chi$-eigenspaces
\begin{equation*}
  0\to (\overline{\mathscr{E}}_\infty/\overline{\mathscr{C}}_\infty(m-1))^\chi \to (\mathscr{U}_\infty/ \overline{\mathscr{C}}_\infty(m-1))^\chi\to \mathscr{X}_\infty(m-1)^\chi \to \mathscr{A}_\infty(m-1)^\chi \to 0.
\end{equation*}

 The following algebraic lemma is similar to the results in \cite[Lemma3.5.5]{C} or \cite[Lemma5.3]{KNF}.
\begin{lemma}\label{lem:herbrand}
 Let $\chi$ be is an irreducible $\mathbb{Z}_p$-representation of $\Delta$ which is nontrivial on the decomposition group of $\mathfrak{p}$ in $\Delta$,
  $\alpha: ((\mathscr{U}_\infty/\overline{\mathscr{C}}_\infty(m-1))^\chi)^\Gamma\to (\mathscr{X}_\infty(m-1))^\chi)^\Gamma$ and
  $\beta: ((\mathscr{U}_\infty/\overline{\mathscr{C}}_\infty(m-1))^\chi)_\Gamma\to (\mathscr{X}_\infty(m-1))^\chi)_\Gamma$ be induced by cohomology. Then kernels
  and cokernels of $\alpha$ and $\beta$ are finite, moreover
 \[ \frac{\# ker \alpha}{\# coker \alpha}  =\frac{\#ker \beta}{\#coker \beta}.\]
\end{lemma}
\begin{proof}
We follow \cite[Lemma5.3]{KNF} to get this case. By Theorem\ref{thm:Imagi-quad-conj-of-Iwasa}(ii), the assumption implies $A:=(\overline{\mathscr{E}}_\infty/\overline{\mathscr{C}}_\infty(m-1))^\chi$ and $B:=\mathscr{A}_\infty(m-1)^\chi$ have the same characteristic ideal, hence $A^\Gamma, A_\Gamma, B^\Gamma, B_\Gamma$ is finite. Moreover, the Herbrand quotients are equal:
\[  \frac{\# A^\Gamma}{\# A_\Gamma}=\frac{\# B^\Gamma}{\# B_\Gamma}. \]
Denote $X=(\mathscr{U}_\infty/\overline{\mathscr{C}}_\infty(m-1))^\chi$ and $Y=\mathscr{X}_\infty(m-1))^\chi$. Consider the following diagram
\[
   \begin{tikzcd}
	0\arrow[r]	& (X/A)^\Gamma \arrow[r] & Y^\Gamma \arrow[r] & B^\Gamma \arrow[r] & (X/A)_\Gamma \arrow[r] & Y_\Gamma \arrow[r] & B_\Gamma\arrow[r] & 0\\
		& X^\Gamma \arrow[r,"="]\arrow[u, "\gamma"'] & X^\Gamma \arrow[u, "\alpha"'] & & X_\Gamma \arrow[r,"="]\arrow[u, "\delta"'] & X_\Gamma \arrow[u, "\beta"'] & &.
	\end{tikzcd}  \]
Take the quotients of kernels and cokernels along the diagram
\[ \frac{\# ker\gamma}{\# coker\gamma}\cdot \frac{1}{\# B^\Gamma} \cdot \frac{\# ker\beta}{\# coker\beta}=\frac{\# ker\alpha}{\# coker\alpha}\cdot
\frac{\# ker\delta}{\# coker\delta}\cdot \frac{1}{\# B_\Gamma} .\]
Since $ker\gamma=A^\Gamma$ and
\[  0 \to coker \gamma \to A_\Gamma \to ker\delta \to 0\]
is exact and $coker\delta=0$, the result follows.
\end{proof}
The following lemma is \cite[Lemma5.3]{KNF}.
\begin{lemma}\label{lem:Zp-extension}
  Let $F_\infty/F$ be any $\mathbb{Z}_p$-extension with Galois group $\Gamma$, and let
  \[ 0\to A \to X \to Y \to B \to 0 \]
  be an exact sequence of $\mathbb{Z}_p[[\Gamma]]$-module, such that $X$ and $Y$ have the same characteristic polynomial. Let $\alpha: X^\Gamma \to Y^\Gamma$ and $\alpha: X_\Gamma \to Y_\Gamma$  be induced by cohomology. If any of the four groups $A^\Gamma, A_\Gamma, B^\Gamma, B_\Gamma$ is finite,
  the same is true for kernels  and cokernels of $\alpha$ and $\beta$ and we have
   \[ \frac{\# ker \alpha}{\# coker \alpha}  =\frac{\#ker \beta}{\#coker \beta}.\]
\end{lemma}

The following theorem is to compare the elliptic units with the \'{e}tale cohomology of $F$.
\begin{theorem}\label{thm:index}
  Let $F$ be an abelian extension field of imaginary quadratic field $K$, and $p$ an odd prime with $p\nmid [F:K]$. Suppose $p$ is split in $K$, or $p$ is inert or ramified over which the elliptic curve $E$ has good reduction. For $m>1$ and for each character $\chi\neq 1$ of
  $\text{Gal}(F/K)$, we have
  $$[H^1(\mathcal{O}_F[1/p],\mathbb{Z}_p(m))^\chi: (\overline{\mathscr{C}}_\infty(m-1)_{\mathscr{G}_F})^\chi]=\frac{\# H^2(\mathcal{O}_F[1/p],\mathbb{Z}_p(m))^\chi}{\#H^0(F,\mathbb{Q}_p/\mathbb{Z}_p(m-1))^\chi}.$$
\end{theorem}
\begin{proof}
\textbf{Case 1.} $p$ is split in $K$. Then we have $(p)=\mathfrak{p}\mathfrak{p}^*$.

  Consider the Poitou-Tate duality sequence
  \begin{align*}
    0 & \to \mathscr{X}(m-1)^{\mathscr{G}_F}\to H^1(\mathcal{O}_F[1/p],\mathbb{Z}_p(m)) \to \bigoplus_{v\in S}\mathscr{X}_{v,\infty}(m-1)_{\mathscr{G}_F} \\
   & \to  \mathscr{X}_{\infty}(m-1)_{\mathscr{G}_F} \to \mathscr{A}'_\infty(m-1)_{\mathscr{G}_F} \to 0,
  \end{align*}
  comparing with the sequence
  \begin{align*}
    0 & \to  (\mathscr{U}_\infty/ \overline{\mathscr{C}}_\infty(m-1))^{\mathscr{G}_F} \to \overline{\mathscr{C}}_\infty(m-1)_{\mathscr{G}_F}\\
     & \to \mathscr{U}_\infty(m-1)_{\mathscr{G}_F} \to  (\mathscr{U}_\infty/ \overline{\mathscr{C}}_\infty(m-1))_{\mathscr{G}_F} \to 0,
  \end{align*}
where $\mathscr{X}_{v,\infty}$ is the Galois group over $F_{v,\infty}:=F_v(\mathfrak{p}^\infty)$ \textcolor{red}{???}of the maximal abelian $p$-extension of $F_{v,\infty}$.

  Let $\mathscr{X}_{p,\infty}= \bigoplus_{v\in S}\mathscr{X}_{v,\infty}.$  Then
  $$(\mathscr{X}_{p,\infty}(m-1))^{\mathscr{G}_F}\cong \bigoplus_{v\in S}H^2(F_v, \mathbb{Q}_p/\mathbb{Z}_p(1-m))^*=0 , $$
 ( where $cd_p(F_v)=1$ since char$(F_v)=p$, see \cite[Theorem7.1.8]{Ne1}  )
  and hence $\mathscr{U}_\infty(m-1)^{\mathscr{G}_F}=0$ because of $\mathscr{U}_\infty(m-1)\subset \mathscr{X}_{p,\infty}(m-1)$.
  Moreover, local class fields theory yields an exact sequence
  $$0\to \mathscr{U}_\infty \to \mathscr{X}_{p,\infty} \to \bigoplus_{v\in S}\mathbb{Z}_p \to 0, \textcolor{red}{\textbf{???}}$$
  which leads to a short exact sequence
  $$ 0\to \mathscr{U}_\infty (m-1)_{\mathscr{G}_F}\to \mathscr{X}_{p,\infty}(m-1)_{\mathscr{G}_F} \to \bigoplus_{v\in S}H^0(F_v, \mathbb{Q}_p/\mathbb{Z}_p(m-1)) \to 0 .$$

 Let us take the $\chi$-eigenspaces of the above sequence. We obtain the following commutative diagram:
 \[
   \begin{tikzcd}
	0\arrow[r]	& (\mathscr{X}(m-1)^\chi)^\Gamma \arrow[r, "u"] & H^1(\mathcal{O}_F[1/p],\mathbb{Z}_p(m))^\chi \\
	0 \arrow[r]	& ((\mathscr{U}_\infty/ \overline{\mathscr{C}}_\infty)(m-1))^\chi)^\Gamma \arrow[r]\arrow[u, "\alpha"'] & (\overline{\mathscr{C}}_\infty(m-1)^\chi)_\Gamma \arrow[u, "\gamma"'] \\
& & 0 \arrow[u]
	\end{tikzcd}  \]
\[
   \begin{tikzcd}
	\arrow[r]&	(\mathscr{X}_{p,\infty}(m-1)^\chi)_\Gamma \arrow[r] & (\mathscr{X}_{\infty}(m-1)^\chi)_\Gamma \arrow[r] & (\mathscr{A}'_\infty(m-1)^\chi)_\Gamma \arrow[r] & 0\\
	 \arrow[r]&	(\mathscr{U}_\infty(m-1)^\chi)_\Gamma \arrow[r]\arrow[u, "\delta"'] & ((\mathscr{U}_\infty/ \overline{\mathscr{C}}_\infty)(m-1))^\chi)_\Gamma  \arrow[r]\arrow[u, "\beta"'] & 0 \\
& 0 \arrow[u]
	\end{tikzcd} \]
For $m>1$, the finiteness of group $H^2(G_S(F),\mathbb{Z}_p(m-1))$ implies the finiteness of $\mathscr{A}'_\infty(m-1)_\Gamma$. By Lemma\ref{lem:Zp-extension}, we know that
$\alpha$ and $\beta$ have finite kernels and cokernels, hence the same is true for coker$\delta$. We obtain
\[   \frac{\#\mbox{ker}\alpha}{\#\mbox{coker}\alpha}\cdot\frac{1}{\#\mbox{coker}\delta}\cdot \frac{1}{\#(\mathscr{A}'_\infty(m-1)^\chi)_\Gamma}=
\frac{1}{\#\mbox{coker}\gamma}\cdot\frac{\#\mbox{ker}\beta}{\#\mbox{coker}\beta}.\]
By Lemma\ref{lem:Zp-extension} and Theorem\ref{thm:Imagi-quad-conj-of-Iwasa}(i), we get
\[ [H^1(\mathcal{O}_F[1/p],\mathbb{Z}_p(m))^\chi:(\overline{\mathscr{C}}_\infty(m-1)^\chi)_\Gamma]=\#\mbox{coker}\delta \cdot \#(\mathscr{A}'_\infty(m-1)^\chi)_\Gamma .\]
There is
\[ (\mathscr{A}'_\infty(m-1)^\chi)_\Gamma\cong \text{\font\fontWCA=wncyr10 {\fontWCA SH}}_S^2(F,\mathbb{Z}_p(m))^\chi, \]
which fits into an exact sequence
\begin{align*}
  0 & \to  \text{\font\fontWCA=wncyr10 {\fontWCA SH}}_S^2(F,\mathbb{Z}_p(m))^\chi \to H^2(\mathcal{O}_F[1/p], \mathbb{Z}_p(m))^\chi\\
   & \to \bigoplus_{v\in S}H^2(F_v, \mathbb{Z}_p(m))^\chi\to (H^0(F,\mathbb{Q}_p/\mathbb{Z}_p(1-m))^*)^\chi \to 0.
\end{align*}
By local duality, $H^2(F_v,\mathbb{Z}_p(m))\cong H^0(F_v,\mathbb{Q}_p/\mathbb{Z}_p(1-m))^*$, hence
\[ H^2(F_v, \mathbb{Z}_p(m))^\chi\cong (H^0(F,\mathbb{Q}_p/\mathbb{Z}_p(1-m))^{\chi^{-1}})^* . \]
Now $\mathbb{Q}_p/\mathbb{Z}_p(1-m)^{\chi^{-1}}=\mathbb{Q}_p/\mathbb{Z}_p^{\chi^{-1}\omega^{m-1}}(1-m)$ is trivial except if $\chi=\omega^{m-1}$, in which case the order of $H^0(F_v,\mathbb{Q}_p/\mathbb{Z}_p(1-m))^*$ is the same as that of $H^0(F_v,\mathbb{Q}_p/\mathbb{Z}_p(m-1))$\textcolor{red}{???}. The same argument
applies to the groups $(H^0(F,\mathbb{Q}_p/\mathbb{Z}_p(1-m))^*)^\chi$ and $H^0(F,\mathbb{Q}_p/\mathbb{Z}_p(m-1)))^\chi$, hence the claim holds.

\textbf{Case 2.} $p$ is inert or ramifies in $K$. If $\chi$ is an irreducible $\mathbb{Z}_p$-representation of $\Delta$ which is nontrivial on the decomposition group of $\mathfrak{p}$ in $\Delta$, then by Theorem\ref{thm:Imagi-quad-conj-of-Iwasa}(ii) and Lemma\ref{lem:herbrand}, through a similar discussion as case 1, we have
$$[H^1(\mathcal{O}_F[1/p],\mathbb{Z}_p(m))^\chi: (\overline{\mathscr{C}}_\infty(m-1)_{\mathscr{G}_F})^\chi]=\frac{\# H^2(\mathcal{O}_F[1/p],\mathbb{Z}_p(m))^\chi}{\#H^0(F,\mathbb{Q}_p/\mathbb{Z}_p(m-1))^\chi}.$$
By \cite[Lemma2.2.10]{Kin01}, we know $\chi$ is non trivial on $\Delta$, where $p$ is a prime over which $E$ has good reduction\textcolor{red}{???}.

\end{proof}

\section{Lichtenbaum conjecture}

Take $\mathfrak{f}$ be the conductor of the abelian extension $F/K$, which is consider as a subfield of $K(\mathfrak{f})$.
\begin{definition}
Let $B_j(K(\mathfrak{f}))$ be the Galois module of $H^1_\mathcal{M}(K(\mathfrak{f}),j+1)$ generated by motivic element $\xi_\chi(j)$, where $\chi$ runs over the character of $G_\mathfrak{f}=\text{Gal}(K(\mathfrak{f})/K)$. Let $B_j(F)\subset H^1_\mathcal{M}(F,j+1)$ be the image of $B_j(K(\mathfrak{f}))$ under corestriction.
\end{definition}

\begin{theorem}
  Let $F=K(\mathfrak{f})$, $E$ be an elliptic curve with CM by $K$. Let $S$ be a set of primes of $2$ and bad primes $p$ over which $E$ has bad reduction. Then the cohomology Lichtenbaum conjecture holds up to the primes in $S$, i.e.,
  \[ \zeta_F(1-m)^*=R_m(F)\prod_p \frac{\# H^2(\mathcal{O}_F[1/p], \mathbb{Z}_p(m))}{\# H^1(\mathcal{O}_F[1/p], \mathbb{Z}_p(m))_{\mbox{tors}}}. \]
\end{theorem}

\begin{proof}
  Let $G=\mbox{Gal}(F/K), \widehat{G}$ be the characters group of $G$. Then we have
  \begin{equation}\label{equ:L-function-euqation}
    \zeta_F(s)=\prod_{\chi\in \widehat{G} }L(\chi,s).
  \end{equation}

  Theorem\ref{Thm3.4} immediately gives the covolume of $\rho_\infty(B_m(K(\mathfrak{f})))$ in $\mathbb{R}(m-1)^+$, that is,
$$\text{covol}(\rho_\infty(B_m(K(\mathfrak{f}))))=\prod_\chi\frac{(2\mathcal{N}\mathfrak{f})^{m-1}\Phi(\mathfrak{f}_\chi)}{(-1)^{1-m}(2m)!\Phi(\mathfrak{f})}
  L'(\chi,1-m).$$
But the Beilinson regulator is
$$R_m(K(\mathfrak{f}))=\frac{\text{covol}(\rho_\infty(B_m(K(\mathfrak{f}))))}{[K_{2m-1}(K(\mathfrak{f}))/\text{tors}:B_m(K(\mathfrak{f}))]}.$$
By \eqref{equ:L-function-euqation}, we have
\[\zeta_F(1-m)^*= R_m(F) \cdot [K_{2m-1}(K(\mathfrak{f}))/\text{tors}:B_m(K(\mathfrak{f}))] \cdot
\prod_\chi\frac{(-1)^{1-m}(2m)!\Phi(\mathfrak{f})}{(2\mathcal{N}\mathfrak{f})^{m-1}\Phi(\mathfrak{f}_\chi)}\]

Theorem\ref{thm:l-adic-regulator} conclude that
\[ [K_{2m-1}(K(\mathfrak{f}))/\text{tors}:B_m(K(\mathfrak{f}))]\sim_p [H^1(\mathcal{O}_F[1/p],\mathbb{Z}_p(m)): (\overline{\mathscr{C}}_\infty(m-1)_{\mathscr{G}_F})], \]
which are equal up to a $p$-adic unit.

From Theorem\ref{thm:index}, we get that
\[\zeta_F(1-m)^*= R_m(F) \cdot\prod_p \frac{\# H^2(\mathcal{O}_F[1/p],\mathbb{Z}_p(m))}{\#H^0(F,\mathbb{Q}_p/\mathbb{Z}_p(m-1))}\cdot
\prod_\chi\frac{(-1)^{1-m}(2m)!\Phi(\mathfrak{f})}{(2\mathcal{N}\mathfrak{f})^{m-1}\Phi(\mathfrak{f}_\chi)}.\]
Lemma2.2 in \cite{KNF} shows that $H^0(F,\mathbb{Q}_p/\mathbb{Z}_p(m-1))=H^1(\mathcal{O}_F[1/p], \mathbb{Z}_p(m))_{\mbox{tors}}$.
\end{proof}

$\mathscr{SUN }$\quad  $\mathscr{ CHAOS}$

\bibliographystyle{amsplain}

\end{document}